\documentclass[11pt,reqno]{amsart}

\usepackage[T1]{fontenc}
\usepackage[utf8]{inputenc}
\usepackage[expansion=false,protrusion=true]{microtype}
\usepackage{amsmath,amssymb,mathtools}
\usepackage{booktabs}
\usepackage{enumitem}
\usepackage{graphicx}
\usepackage[margin=2.6cm]{geometry}
\usepackage[colorlinks=true,linkcolor=black,citecolor=black,urlcolor=black]{hyperref}

\numberwithin{equation}{section}

\theoremstyle{plain}
\newtheorem{theorem}{Theorem}[section]
\newtheorem{proposition}[theorem]{Proposition}
\newtheorem{lemma}[theorem]{Lemma}
\newtheorem{corollary}[theorem]{Corollary}
\theoremstyle{definition}
\newtheorem{assumption}[theorem]{Assumption}
\theoremstyle{remark}
\newtheorem{remark}[theorem]{Remark}

\newcommand{\btau}{\boldsymbol{\tau}}
\newcommand{\bY}{\mathbf{Y}}
\newcommand{\bU}{\mathbf{U}}
\newcommand{\bX}{\mathbf{X}}
\newcommand{\bW}{\mathbf{W}}
\newcommand{\bS}{\mathbf{S}}
\newcommand{\bone}{\mathbf{1}}
\newcommand{\be}{\mathbf{e}}
\newcommand{\bze}{\mathbf{0}}
\newcommand{\bx}{\mathbf{x}}

\newcommand{\bnu}{\boldsymbol{\nu}}
\newcommand{\tH}{\widetilde H}
\newcommand{\Sig}{\Sigma_{\tH}}

\newcommand{\Hmax}{\tH_{\max}}
\newcommand{\Hmin}{\tH_{\min}}
\newcommand{\R}{\mathbb{R}}
\newcommand{\N}{\mathbb{N}}
\newcommand{\Prob}{\mathbb{P}}
\newcommand{\E}{\mathbb{E}}
\newcommand{\Norm}{\mathcal{N}}
\newcommand{\Hi}{\mathfrak{H}}
\DeclareMathOperator{\Var}{Var}
\DeclareMathOperator{\Cov}{Cov}
\DeclareMathOperator{\tr}{tr}
\DeclareMathOperator{\rank}{rank}
\DeclareMathOperator{\MSE}{MSE}

\begin{document}

\title[Exact inference for multi-sub-fractional Brownian motion]
      {Exact maximum likelihood inference for drifted\\
       multi-sub-fractional Brownian motion\\
       at discrete observation}
\author{Afrah Al-Harby}
\address{Department of Mathematics, College of Science, Imam Abdulrahman Bin Faisal
University, P.\,O.\ Box 1982, Dammam, Saudi Arabia}
\email{2250500238@iau.edu.sa}

\author{Ezzedine Mliki}
\address{Department of Mathematics, College of Science, Imam Abdulrahman Bin Faisal
University, P.\,O.\ Box 1982, Dammam, Saudi Arabia}
\address{Basic and Applied Scientific Research Center, Imam Abdulrahman Bin Faisal
University, P.\,O.\ Box 1982, 31441, Dammam, Saudi Arabia}
\email{ermliki@iau.edu.sa}

\author{Manal Al-Ohali}
\address{Department of Mathematics, College of Science, Imam Abdulrahman Bin Faisal
University, P.\,O.\ Box 1982, Dammam, Saudi Arabia}
\email{malohali@iau.edu.sa}

\begin{abstract}
Sub-fractional Brownian motion is self-similar and long-range dependent but has no
stationary increments, so the increment covariance is not Toeplitz and no spectral
density is available. We show that a complete finite-sample likelihood theory survives
nonetheless. The model is a linear trend observed at $N$ equidistant times through a
superposition of $m$ independent sub-fractional Brownian motions with known Hurst indices
and a common scale. Nondegeneracy follows from realising the process as the even part of
a two-sided fractional Brownian motion, and the maximum likelihood estimators of the
trend and of the scale are explicit. The statistics on which inference rests are pivotal,
their laws depending on the sample size alone, so intervals and tests of exact level are
available at every $N\ge2$, together with complete sufficiency, minimum variance
unbiasedness and attainment of the Cram\'er--Rao bound. An explicit variance bound gives
strong consistency and asymptotic normality, and simulations confirm the exact coverage
and the predicted effect of a misspecified Hurst vector.
\end{abstract}


\maketitle

\section{Introduction}

Sub-fractional Brownian motion was introduced by Bojdecki, Gorostiza and Talarczyk
\cite{BGT04} as the centred Gaussian process $S^{H}=\{S^{H}(t),\,t\ge0\}$ with
$S^{H}(0)=0$ and covariance function
\begin{equation}\label{eq:sfbmcov}
  R_{H}(s,t)=\E\bigl[S^{H}(t)S^{H}(s)\bigr]
  =s^{2H}+t^{2H}-\tfrac12\bigl[(s+t)^{2H}+|s-t|^{2H}\bigr],\qquad s,t\ge0,
\end{equation}
where $H\in(0,1)$. It arose as the limit of occupation time fluctuations of a system of
independent particles and it interpolates, in a precise sense, between Brownian motion
and fractional Brownian motion: like fractional Brownian motion it is self-similar of
index $H$, has H\"older paths of any order less than $H$, is not a semimartingale for
$H\neq\frac12$, and reduces to Brownian motion at $H=\frac12$. Unlike fractional
Brownian motion, however, \emph{its increments are not stationary}. They are
nonstationary in a controlled way: by \eqref{eq:sfbmcov} the increment variance satisfies
\begin{equation}\label{eq:incrbounds}
  \bigl(2-2^{2H-1}\bigr)(t-s)^{2H}\le\E\bigl[(S^{H}(t)-S^{H}(s))^{2}\bigr]
  \le(t-s)^{2H},\qquad H>\tfrac12,
\end{equation}
with the inequalities reversed for $H<\frac12$, and the covariance between increments
decays faster than
those of fractional Brownian motion, so that the long-range dependence is weaker. Path
and distributional properties were developed further by Tudor \cite{Tud07} and the local
time was studied by Mendy \cite{Men10}. Two features account for its use in applications. The first is that the dependence is
weaker than for fractional Brownian motion and decays with the distance from the origin
of time, which suits records in which long memory is present but attenuates as the
process moves away from its starting point; a stationary-increment model would overstate
the persistence of the early part of such a record. The second is that the process
remains Gaussian and self-similar, so models built on it stay tractable. In mathematical
finance it has been used for the pricing of geometric Asian power options \cite{WCT21},
where the sub-fractional kernel reproduces departures from the Black--Scholes price while
the self-similarity is retained. In statistical physics it serves as a Gaussian model of
anomalous diffusion: by \eqref{eq:sfbmcov} the variance is
$(2-2^{2H-1})t^{2H}$, subdiffusive for $H<\frac12$ and superdiffusive for $H>\frac12$,
which is close to the setting in which the process first arose, as the limit of
occupation time fluctuations of a system of independent particles \cite{BGT04}. The same
covariance is a natural candidate for environmental and network-traffic series, where
the intensity of the dependence is itself a function of the time elapsed since the start
of the observation.

A single fractional component, sub-fractional or not, may be insufficient for phenomena
that exhibit several scaling regimes at once. This motivates superpositions of finitely
many independent components with distinct Hurst parameters. For fractional Brownian
motion the resulting class was studied by Maleki Almani and Sottinen \cite{MS23}; the
sub-fractional counterpart, the multi-mixed sub-fractional Brownian motion, was
introduced and analysed by Shokrollahi, Sottinen and Zili \cite{SSZ26}. Related
constructions and their correlation structure appear in Alajmi and Mliki \cite{AM21},
Sghir \cite{Sgh13} and Zili \cite{Zil18}.

Parameter estimation for these processes is a separate matter, and the literature is
asymptotic. Likelihood inference for self-similar Gaussian processes observed on a grid
goes back to Dahlhaus \cite{Dah89}, who established efficiency of the maximum likelihood
estimator in the stationary-increment case. For the drifted fractional Brownian motion
$Y(t)=\mu t+\sigma B^{H}(t)$
observed at discrete times, Hu, Nualart, Xiao and Zhang \cite{HNXZ11} obtained
closed-form maximum likelihood estimators and their limit theory. Mishura, Ralchenko and
Shklyar \cite{MRS17} extended the likelihood approach to Gaussian processes with
stationary increments, Xiao, Zhang and Zhang \cite{XZZ11} treated the mixed fractional
model, and Mishura and Voronov \cite{MV15} the sum of two fractional Brownian motions.
Ralchenko and Yakovliev \cite{RY23,RY24} estimate \emph{two} separate scales and prove
joint asymptotic normality, at the price of the restrictions $H<\frac12$ and
$H_1<H_2<\frac34$ respectively. For the sub-fractional model the reference is Kuang and
Liu \cite{KL15}, who estimate $(\mu,\sigma^{2})$ for a \emph{single} sub-fractional
Brownian motion with drift at discrete observation and obtain strong consistency, a
central limit theorem and Berry--Ess\'een bounds, the last being needed precisely because
the finite-sample law is unavailable.

All of this concerns the increasing-domain regime, in which the mesh is held fixed and
the observation window grows. The complementary infill regime, where the mesh shrinks on
a fixed horizon, has been developed for mixed models in a parallel literature: Mies and
Podolskij \cite{MP23} construct high-frequency estimators for mixed fractional stable
processes, Chong, Delerue and Mies \cite{CDM25} obtain the optimal rates for the joint
estimation of the Hurst index and the scales of a mixed semimartingale together with
matching lower bounds, and Chong, Delerue and Li \cite{CDL25} analyse the effect of
fractional noise in high-frequency data. The two regimes see different parameters. Under
infill sampling it is the Hurst indices that the data determine and the drift that they
do not; here the Hurst vector is taken as known and the drift is what can be recovered.
Proposition \ref{pr:varbound} below makes the second half of that statement precise: the
accuracy of $\widehat\mu$ is governed by the length $t_N=Nh$ of the observation window
and not by the mesh, so refining the grid on a fixed horizon does not help the drift at
all.

So the process class of \cite{SSZ26} exists and its path properties are understood, the
one-component estimation problem has been solved asymptotically, and the
multi-component estimation problem has not been addressed. The present paper addresses
it, and shows that with a single scale shared by the components the inference is not
asymptotic but exact.

It is worth saying at once which part of that is classical. Once the covariance is known
up to a single positive constant, whitening reduces the model to the Gaussian linear
model with one regressor (Remark \ref{re:glm}), and the closed forms, the exact joint
law, the sufficiency, the minimum variance unbiasedness and the exact intervals of
Section \ref{se:mle} are the classical theory of that model transported through
$\Sig^{-1/2}$; the point is already implicit for the fractional model in \cite{HNXZ11}.
What is not classical is that the reduction is available at all here. It requires the
nondegeneracy of $\Sig$, which the loss of stationary increments turns from an
elementary computation into a question about local nondeterminism (Lemma
\ref{le:nondeg}), and what the exactness then delivers is governed entirely by the
scalar $\kappa_N$, whose behaviour is the subject of Section \ref{se:asymp} and is
specific to the sub-fractional kernel.

\subsection*{The observation scheme and what we prove}
Fix a mesh $h>0$ and sample at $t_k=kh$, $k=1,\dots,N$. The data are the vector
\begin{equation}\label{eq:obsvec}
  \bY=\bigl(Y(t_1),\dots,Y(t_N)\bigr)^{\top}
     =\mu\btau+\beta\sum_{r=1}^{m}\bS^{\tH_r},
\end{equation}
generated by the continuous-time model
\begin{equation}\label{eq:model}
  Y(t)=\mu t+\beta\sum_{r=1}^{m}S^{\tH_r}(t),\qquad t\ge0,
\end{equation}
where $S^{\tH_1},\dots,S^{\tH_m}$ are independent sub-fractional Brownian motions with
known Hurst parameters $\tH_1,\dots,\tH_m\in(1/2,1)$, and
\[
  \btau=(t_1,\dots,t_N)^{\top},\qquad
  \bS^{\tH_r}=\bigl(S^{\tH_r}(t_1),\dots,S^{\tH_r}(t_N)\bigr)^{\top}.
\]
Under independence, $\bY$ is Gaussian with mean $\mu\btau$ and covariance matrix
$\beta^{2}\Sig$, where $\Sig=\sum_{r=1}^{m}\Sigma_{\tH_r}$ and
\begin{equation}\label{eq:GammaHr}
  (\Sigma_{\tH_r})_{ij}
  =t_i^{2\tH_r}+t_j^{2\tH_r}
   -\tfrac12\bigl[(t_i+t_j)^{2\tH_r}+|t_i-t_j|^{2\tH_r}\bigr],
  \qquad i,j=1,\dots,N .
\end{equation}
Model \eqref{eq:model} is the common-scale specialisation of
$Y(t)=\mu t+\sum_r\beta_rS^{\tH_r}(t)$. The restriction is what makes the likelihood
explicitly solvable, because the covariance matrix is then a \emph{known} matrix times a
single unknown constant; it is natural whenever the several sources of dependence are
believed to act at a common intensity and to differ only in memory, and it is what the
classical mixed models reduce to after reparameterisation.

The maximum likelihood estimators $\widehat\mu$ and $\widehat\beta^{\,2}$ are available in
closed form (Theorem \ref{th:mle}), and the whole multi-component structure enters the
inference through the single scalar $\kappa_N=\btau^{\top}\Sig^{-1}\btau$. The
nonsingularity of $\Sig$ that this requires is established in Lemma \ref{le:nondeg} by
representing sub-fractional Brownian motion as the even part of a two-sided fractional
Brownian motion, the elementary arguments available in the stationary-increment case
being of no use here. The joint law of $(\widehat\mu,\widehat\beta^{\,2})$ is then exact
at every sample size: the two estimators are independent, $\widehat\mu$ is Gaussian and
$N\widehat\beta^{\,2}/\beta^{2}$ is chi-square with $N-1$ degrees of freedom
(Theorem \ref{th:exactlaw}). Exact confidence intervals and exact tests follow
(Corollaries \ref{co:ci} and \ref{co:tests}), as do complete sufficiency, uniform minimum
variance unbiasedness and attainment of the Cram\'er--Rao bound (Theorem \ref{th:umvu},
Proposition \ref{pr:fisher}).

Beyond the exact theory we establish an explicit non-asymptotic bound on the variance of
the drift estimator (Proposition \ref{pr:varbound}), showing that the accuracy of drift
estimation is governed by the length $t_N=Nh$ of the observation window and by
$\Hmax=\max_r\tH_r$, with the same constant as in the fractional case. Both estimators
are strongly consistent (Theorems \ref{th:scmu} and \ref{th:scbeta}) and asymptotically
normal (Theorem \ref{th:an}), and the sequence $\{\widehat\mu^{(N)}\}$ is, in law, a
Brownian motion evaluated along the decreasing sequence of its own variances
(Theorem \ref{th:indincr} and Corollary \ref{co:tcbm}), which yields a second,
moment-free proof of strong consistency. A Monte Carlo study confirms the exact
distributional results and the exact coverage of the confidence intervals, and measures
the cost of a misspecified Hurst vector (Section \ref{se:sim}).

Asymptotic statements always refer to a fixed mesh and a growing window: $h>0$ stays put
while $N\to\infty$, so that $t_N=Nh\to\infty$. Nothing below is an infill statement.

\section{The model and its likelihood}

\subsection{Standing assumptions}

\begin{assumption}\label{as:main}
We observe \eqref{eq:model} at the $N\ge2$ equidistant times $t_k=kh$. The components
$S^{\tH_1},\dots,S^{\tH_m}$ are independent sub-fractional Brownian motions, their number
$m\in\N$, their Hurst parameters $\tH_r\in(0,1)$ and the mesh $h>0$ being fixed and
known; the drift $\mu\in\R$ and the scale $\beta>0$ are the parameters to be
estimated. The restriction $\tH_r>\frac12$, natural for the long-memory reading of the
model, is not needed anywhere in Sections \ref{se:mle} and \ref{se:asymp} except in the
superadditivity step of Proposition \ref{pr:varbound}; for $\tH_r<\frac12$ the cruder
elementwise bound $(\Sig)_{ij}\le t_i^{2\tH_r}+t_j^{2\tH_r}$ gives the same rate with the
constant $9$ in place of $9/2$.
\end{assumption}

\begin{remark}\label{re:ident}
Only $\Sig$ enters the likelihood, so only $\Sig$ is identified, and the
parameterisation by $(m,\tH)$ is not: $m$ components sharing a common $\tH$ and scale
$\beta$ give the same law as one component with scale $\beta\sqrt m$. Nothing in the
inference is affected, since every statement below is about $\Sig$ and $\kappa_N$; the
$\tH_r$ may therefore be assumed distinct without loss of generality, and we do so
whenever it is convenient.
\end{remark}

We write
\[
  \Hmax=\max_{1\le r\le m}\tH_r,\qquad \Hmin=\min_{1\le r\le m}\tH_r,\qquad
  S^{\tH}=\sum_{r=1}^{m}S^{\tH_r},\qquad
  \bU=\bY-\mu\btau=\beta\,\bS^{\tH} .
\]
The letter $C$ denotes a finite positive constant that may change from line to line and
depends only on $h$, $m$, $\tH$, $\beta$ and, where indicated, on a moment order $q$.

\subsection{The covariance structure}

\begin{lemma}\label{le:law}
Under Assumption \ref{as:main}, $\bY\sim\Norm(\mu\btau,\beta^{2}\Sig)$ with
$\Sig=\sum_{r=1}^{m}\Sigma_{\tH_r}$ and $\Sigma_{\tH_r}$ given by \eqref{eq:GammaHr}.
\end{lemma}

\begin{proof}
Each $\bS^{\tH_r}$ is a centred Gaussian vector with covariance $\Sigma_{\tH_r}$ by
\eqref{eq:sfbmcov}. The components are independent, so $\bS^{\tH}=\sum_r\bS^{\tH_r}$ is
centred Gaussian with covariance $\sum_r\Sigma_{\tH_r}=\Sig$, and $\bY$ is the affine
image $\mu\btau+\beta\bS^{\tH}$.
\end{proof}

\begin{remark}\label{re:nostat}
Two features distinguish this model from its fractional counterpart and they govern the
rest of the paper.

First, $S^{\tH}$ does \emph{not} have stationary increments. By \eqref{eq:sfbmcov}, for
$0\le s\le t$,
\[
  \E\bigl[(S^{H}(t)-S^{H}(s))^{2}\bigr]
  =-2^{2H-1}(t^{2H}+s^{2H})+(t+s)^{2H}+(t-s)^{2H},
\]
which depends on $s$ and $t$ separately and not only on $t-s$. Consequently the increment
vector with coordinates $S^{\tH}(t_k)-S^{\tH}(t_{k-1})$ has a covariance matrix that is
\emph{not} Toeplitz, there is no reduction of $\kappa_N=\btau^{\top}\Sig^{-1}\btau$ to a quadratic
form in the inverse of a Toeplitz matrix, and there is no spectral density to which the
classical theory of the sample mean of a stationary sequence could be applied. Every
statement below about $\kappa_N$ is therefore obtained from the covariance matrix itself.

Second, $S^{\tH}$ is not self-similar when the $\tH_r$ are not all equal, since
$\Var S^{\tH}(t)=\sum_r(2-2^{2\tH_r-1})t^{2\tH_r}$ is not a single power of $t$. It is
this superposition of scaling exponents that gives the model its multi-scale character.
\end{remark}

\subsection{Nondegeneracy}

The likelihood requires $\Sig$ to be invertible and the canonical reduction of
Section \ref{se:mle} requires $\Sig^{-1/2}$. In the fractional case nonsingularity is
classical. Here we use the following representation, which is immediate from
\eqref{eq:sfbmcov} and appears in \cite{BGT04}: if $\{B^{H}(u),\,u\in\R\}$ is a two-sided
fractional Brownian motion with Hurst parameter $H$, then
\begin{equation}\label{eq:evenpart}
  \Bigl\{\tfrac{1}{\sqrt2}\bigl(B^{H}(t)+B^{H}(-t)\bigr),\,t\ge0\Bigr\}
  \overset{d}{=}\bigl\{S^{H}(t),\,t\ge0\bigr\} .
\end{equation}
Sub-fractional Brownian motion is, in law, the even part of a two-sided fractional
Brownian motion.

\begin{lemma}\label{le:nondeg}
Under Assumption \ref{as:main} the matrix $\Sig$ is symmetric positive definite. In
particular $\Sig^{-1}$ and $\Sig^{-1/2}$ exist, the latter being the inverse of the unique
symmetric positive definite square root of $\Sig$.
\end{lemma}

\begin{proof}
Being a covariance matrix, $\Sig$ is symmetric and positive semidefinite, so only strict
definiteness has to be proved. Since the components are independent,
$\Sig=\sum_{r}\Sigma_{\tH_r}$ with every summand positive semidefinite, so
$\bx^{\top}\Sig\bx\ge\bx^{\top}\Sigma_{\tH_1}\bx$ for every $\bx\in\R^{N}$ and it suffices
to treat a single component. Write $H=\tH_1$ and $S=S^{H}$.

Let $\bx=(x_1,\dots,x_N)^{\top}\neq\bze$ and suppose $\bx^{\top}\Sigma_{H}\bx=0$. Since
$S$ is centred,
\[
  \bx^{\top}\Sigma_{H}\bx
  =\Var\Bigl(\sum_{i=1}^{N}x_iS(t_i)\Bigr)
  =\E\Bigl[\Bigl(\sum_{i=1}^{N}x_iS(t_i)\Bigr)^{2}\Bigr],
\]
so that $\sum_{i}x_iS(t_i)=0$ almost surely. By \eqref{eq:evenpart} this is equivalent to
\[
  \sum_{i=1}^{N}x_i\bigl(B^{H}(t_i)+B^{H}(-t_i)\bigr)=0\qquad\text{a.s.},
\]
a linear combination of the two-sided fractional Brownian motion $B^{H}$ at the $2N$
points $\pm t_1,\dots,\pm t_N$, which are pairwise distinct and nonzero because
$0<t_1<\dots<t_N$. Fractional Brownian motion is strongly locally nondeterministic
\cite[Sec.~3]{Pit78}, hence its covariance matrix at finitely many distinct nonzero
points is nonsingular, and therefore all the coefficients vanish: $x_i=0$ for every $i$,
contradicting $\bx\neq\bze$.

Hence $\bx^{\top}\Sigma_{H}\bx>0$ for every $\bx\neq\bze$, so $\Sigma_{\tH_1}$, and with
it $\Sig$, is positive definite. A symmetric positive definite matrix has a unique
symmetric positive definite square root and both it and the matrix itself are invertible.
\end{proof}

\begin{remark}\label{re:slnd}
An alternative proof replaces \eqref{eq:evenpart} by the strong local nondeterminism of
$S^{H}$ itself on compact intervals of $[0,\infty)$, which gives
$\Var\bigl(S(t_k)\mid S(t_1),\dots,S(t_{k-1})\bigr)\ge c\,h^{2H}>0$ and the same
contradiction. We prefer \eqref{eq:evenpart} because it reduces the question to the
classical fractional case.
\end{remark}

\subsection{The likelihood function}

\begin{theorem}\label{th:loglik}
Under Assumption \ref{as:main} the log-likelihood of $(\mu,\beta^{2})$ given $\bY$ is
\begin{equation}\label{eq:loglik}
  \ell(\mu,\beta^{2};\bY)
  =-\frac{N}{2}\log(2\pi)-\frac{N}{2}\log\beta^{2}-\frac12\log|\Sig|
   -\frac{1}{2\beta^{2}}(\bY-\mu\btau)^{\top}\Sig^{-1}(\bY-\mu\btau).
\end{equation}
\end{theorem}

\begin{proof}
By Lemma \ref{le:law} the density of $\bY$ is
$(2\pi\beta^{2})^{-N/2}|\Sig|^{-1/2}
\exp\{-\tfrac{1}{2\beta^{2}}(\bY-\mu\btau)^{\top}\Sig^{-1}(\bY-\mu\btau)\}$, which is
well defined by Lemma \ref{le:nondeg}; take logarithms.
\end{proof}

\section{The estimators and their exact law}\label{se:mle}

Throughout this section we write
\begin{equation}\label{eq:dN}
  \kappa_N=\btau^{\top}\Sig^{-1}\btau>0 .
\end{equation}

\subsection{Closed forms}

\begin{theorem}\label{th:mle}
Under Assumption \ref{as:main} the log-likelihood \eqref{eq:loglik} has the unique
maximiser
\begin{equation}\label{eq:mle}
  \widehat\mu=\frac{\btau^{\top}\Sig^{-1}\bY}{\btau^{\top}\Sig^{-1}\btau},
  \qquad
  \widehat\beta^{\,2}
  =\frac1N\left[\bY^{\top}\Sig^{-1}\bY
     -\frac{\bigl(\btau^{\top}\Sig^{-1}\bY\bigr)^{2}}{\btau^{\top}\Sig^{-1}\btau}\right].
\end{equation}
\end{theorem}

\begin{proof}
For fixed $\beta^{2}$ the map $\mu\mapsto\ell$ is a concave quadratic, and
$\partial_\mu\ell=\beta^{-2}\btau^{\top}\Sig^{-1}(\bY-\mu\btau)$ vanishes exactly at
$\widehat\mu$, which is well defined because $\kappa_N>0$. Substituting and writing
$Q=(\bY-\widehat\mu\btau)^{\top}\Sig^{-1}(\bY-\widehat\mu\btau)$, the profile
log-likelihood is $-\frac{N}{2}\log\beta^{2}-Q/(2\beta^{2})$ plus constants, which is
maximised at $\beta^{2}=Q/N$. Expanding $Q$ gives the second formula in \eqref{eq:mle}.
\end{proof}

\begin{remark}\label{re:gls}
$\widehat\mu$ is the generalized least squares estimator of the slope in the
regression of $\bY$ on $\btau$ with error covariance $\Sig$, and $N\widehat\beta^{\,2}$ is
the corresponding residual sum of squares in the metric $\Sig^{-1}$. Only the scalar $\kappa_N$
and the two inner products $\btau^{\top}\Sig^{-1}\bY$ and $\bY^{\top}\Sig^{-1}\bY$ have to be
computed, all three from a single Cholesky factorisation of $\Sig$.
\end{remark}

\subsection{Whitening}

\begin{lemma}\label{le:canon}
Let $\bX=\beta^{-1}\Sig^{-1/2}\bU=\Sig^{-1/2}\bS^{\tH}$ and
\[
  \be=\frac{\Sig^{-1/2}\btau}{\sqrt{\kappa_N}},\qquad \Pi=I_N-\be\be^{\top} .
\]
Then $\bX\sim\Norm(\bze,I_N)$, $\|\be\|=1$, $\Pi$ is the orthogonal projection onto
$\be^{\perp}$ with $\rank\Pi=N-1$, and
\begin{equation}\label{eq:canon}
  \widehat\mu-\mu=\frac{\beta}{\sqrt{\kappa_N}}\,\be^{\top}\bX,
  \qquad
  \frac{N\widehat\beta^{\,2}}{\beta^{2}}=\bX^{\top}\Pi\bX .
\end{equation}
\end{lemma}

\begin{proof}
$\bS^{\tH}\sim\Norm(\bze,\Sig)$, so $\bX=\Sig^{-1/2}\bS^{\tH}\sim\Norm(\bze,I_N)$, and
$\|\be\|^{2}=\btau^{\top}\Sig^{-1}\btau/\kappa_N=1$. Substituting
$\bY=\mu\btau+\beta\Sig^{1/2}\bX$ into \eqref{eq:mle},
\[
  \widehat\mu-\mu
  =\frac{\beta\,\btau^{\top}\Sig^{-1/2}\bX}{\kappa_N}
  =\frac{\beta}{\sqrt{\kappa_N}}\,\be^{\top}\bX,
\]
and
\[
  \frac{N\widehat\beta^{\,2}}{\beta^{2}}
  =\bX^{\top}\bX-\frac{(\btau^{\top}\Sig^{-1/2}\bX)^{2}}{\kappa_N}
  =\bX^{\top}\bX-(\be^{\top}\bX)^{2}=\bX^{\top}\Pi\bX .
\]
Finally $\Pi^{2}=\Pi=\Pi^{\top}$ and $\tr\Pi=N-1$.
\end{proof}

\begin{remark}\label{re:glm}
Whitening by $\Sig^{-1/2}$ turns the model into a Gaussian linear model with one
regressor. Writing $\bW=\Sig^{-1/2}\bY$, which is observable, we have
\begin{equation}\label{eq:glm}
  \bW\sim\Norm\bigl(\mu\sqrt{\kappa_N}\,\be,\ \beta^{2}I_N\bigr),
\end{equation}
with $\be$ as in Lemma \ref{le:canon}: a known unit regressor $\be$, the coefficient
$\mu\sqrt{\kappa_N}$ and the unknown error variance $\beta^{2}$. The vector $\bX$ of
Lemma \ref{le:canon} is $\beta^{-1}$ times the centred version of $\bW$, hence standard
normal; being a function of $\beta$ it is not itself a statistic. All the exact distribution
theory of Section \ref{se:mle} is that of this classical model, transported back through
$\Sig^{1/2}$. What the model contributes is the covariance $\Sig$ and the scalar $\kappa_N$ it
produces; what is specific to the sub-fractional case is everything that concerns $\kappa_N$
itself, and that is the content of Section \ref{se:asymp}. The same
remark identifies the boundary of the theory: with component-specific scales
$\beta_1,\dots,\beta_m$ the covariance $\sum_r\beta_r^{2}\Sigma_{\tH_r}$ is no longer a
known matrix times an unknown constant, the whitening is no longer available, and
exactness is lost.
\end{remark}

\subsection{The joint law at every sample size}

\begin{theorem}\label{th:exactlaw}
Under Assumption \ref{as:main}, for every $N\ge2$:
\begin{enumerate}[label=\textup{(\roman*)}]
\item $\widehat\mu\sim\Norm\bigl(\mu,\ \beta^{2}/\kappa_N\bigr)$; in particular
      $\widehat\mu$ is unbiased;
\item $N\widehat\beta^{\,2}/\beta^{2}\sim\chi^{2}_{N-1}$;
\item $\widehat\mu$ and $\widehat\beta^{\,2}$ are independent.
\end{enumerate}
\end{theorem}

\begin{proof}
By Lemma \ref{le:canon}, $\be^{\top}\bX\sim\Norm(0,1)$ because $\|\be\|=1$, which gives
(i). Since $\Pi$ is an orthogonal projection of rank $N-1$ and $\bX\sim\Norm(\bze,I_N)$,
Cochran's theorem gives $\bX^{\top}\Pi\bX\sim\chi^{2}_{N-1}$, which is (ii). Finally
$\Pi\be=\bze$, so the Gaussian vector $\Pi\bX$ and the Gaussian variable $\be^{\top}\bX$ have
covariance $\Pi\be=\bze$ and are therefore independent; $\widehat\beta^{\,2}$ is a function
of $\Pi\bX$ and $\widehat\mu$ a function of $\be^{\top}\bX$, which is (iii).
\end{proof}

\begin{remark}\label{re:pivot}
Part (ii) is remarkable in that the law of $N\widehat\beta^{\,2}/\beta^{2}$ depends on
$N$ alone. It does not depend on $\mu$, on $\beta$, on the mesh $h$, on the number of
components $m$, or on the Hurst vector $\tH$. The whitening removes the entire
dependence structure from the sampling distribution of the scale estimator. The same is
true of $\sqrt{\kappa_N}(\widehat\mu-\mu)/\beta$, which is standard normal whatever the
model. The whole model enters the sampling distributions only through the single scalar
$\kappa_N$, and only through the drift.
\end{remark}

\begin{corollary}\label{co:moments}
For every $N\ge2$,
\[
  \E\bigl[\widehat\beta^{\,2}\bigr]=\frac{N-1}{N}\beta^{2},\qquad
  \Var\bigl(\widehat\beta^{\,2}\bigr)=\frac{2(N-1)}{N^{2}}\beta^{4},\qquad
  \MSE\bigl(\widehat\beta^{\,2}\bigr)=\frac{2N-1}{N^{2}}\beta^{4},
\]
so that $\widetilde\beta^{\,2}=\frac{N}{N-1}\widehat\beta^{\,2}$ is unbiased with
$\Var(\widetilde\beta^{\,2})=2\beta^{4}/(N-1)$. Moreover
$\E[\widehat\beta]=c_N\beta$ with
$c_N=\sqrt{2/N}\,\Gamma(N/2)/\Gamma((N-1)/2)$, and $c_N\uparrow1$.
\end{corollary}

\begin{proof}
If $Q\sim\chi^{2}_{N-1}$ then $\E Q=N-1$, $\Var Q=2(N-1)$ and
$\E\sqrt{Q}=\sqrt2\,\Gamma(N/2)/\Gamma((N-1)/2)$. Apply these to
$Q=N\widehat\beta^{\,2}/\beta^{2}$ and use
$\widehat\beta=\beta\sqrt{Q}/\sqrt N$; the mean square error is the variance plus the
squared bias $\beta^{4}/N^{2}$.
\end{proof}

\subsection{Optimality}

\begin{theorem}\label{th:umvu}
Under Assumption \ref{as:main}, with $\psi=(\mu,\beta^{2})$ ranging over
$\R\times(0,\infty)$:
\begin{enumerate}[label=\textup{(\roman*)}]
\item the family of laws of $\bY$ is a full-rank two-parameter exponential family with
      natural sufficient statistic
      $\bigl(\btau^{\top}\Sig^{-1}\bY,\ \bY^{\top}\Sig^{-1}\bY\bigr)$, which is therefore
      complete and sufficient;
\item $\widehat\mu$ is the uniformly minimum variance unbiased estimator of $\mu$;
\item $\widetilde\beta^{\,2}=\frac{N}{N-1}\widehat\beta^{\,2}$ is the uniformly minimum
      variance unbiased estimator of $\beta^{2}$, and $\widehat\beta/c_N$ is that of
      $\beta$.
\end{enumerate}
\end{theorem}

\begin{proof}
Writing the density from Theorem \ref{th:loglik} in exponential form,
\[
  f(\bY;\psi)=\exp\Bigl\{\eta_1\,\btau^{\top}\Sig^{-1}\bY
  +\eta_2\,\bY^{\top}\Sig^{-1}\bY-A(\eta)\Bigr\}g(\bY),
  \qquad \eta_1=\frac{\mu}{\beta^{2}},\quad\eta_2=-\frac{1}{2\beta^{2}},
\]
and as $\psi$ ranges over $\R\times(0,\infty)$ the natural parameter
$(\eta_1,\eta_2)$ ranges over $\R\times(-\infty,0)$, which contains a two-dimensional
open rectangle. The family is therefore of full rank and the natural statistic is
complete and sufficient \cite[Thm.~1.6.22 and Thm.~4.3.1]{LC98}. This gives (i). The
estimators $\widehat\mu$ and $\widehat\beta^{\,2}$ are functions of that statistic,
and by Theorem \ref{th:exactlaw} and Corollary \ref{co:moments} the quantities
$\widehat\mu$, $\widetilde\beta^{\,2}$ and $\widehat\beta/c_N$ are unbiased for
$\mu$, $\beta^{2}$ and $\beta$. The Lehmann--Scheff\'e theorem
\cite[Thm.~2.1.11]{LC98} identifies each as the unique uniformly minimum variance
unbiased estimator of the corresponding parameter.
\end{proof}

\begin{proposition}\label{pr:fisher}
Under Assumption \ref{as:main} the Fisher information matrix of $\psi=(\mu,\beta^{2})$
contained in $\bY$ is
\begin{equation}\label{eq:fisher}
  I_N(\psi)=\begin{pmatrix} \kappa_N/\beta^{2} & 0\\[2pt] 0 & N/(2\beta^{4})\end{pmatrix}.
\end{equation}
Consequently:
\begin{enumerate}[label=\textup{(\roman*)}]
\item $\mu$ and $\beta^{2}$ are orthogonal parameters;
\item $\Var(\widehat\mu)=\beta^{2}/\kappa_N$, so $\widehat\mu$ attains the
      Cram\'er--Rao bound at every sample size $N\ge2$, not merely asymptotically;
\item the Cram\'er--Rao bound for unbiased estimation of $\beta^{2}$ is
      $2\beta^{4}/N$; it is attained by no unbiased estimator, and
      $\widetilde\beta^{\,2}$ exceeds it by the factor $N/(N-1)\to1$.
\end{enumerate}
\end{proposition}

\begin{proof}
For $\bY\sim\Norm(\bnu(\psi),\Sigma(\psi))$ the information matrix has entries
$I_{ij}=(\partial_i\bnu)^{\top}\Sigma^{-1}(\partial_j\bnu)
+\frac12\tr(\Sigma^{-1}(\partial_i\Sigma)\Sigma^{-1}(\partial_j\Sigma))$. Here
$\bnu=\mu\btau$ and $\Sigma=\beta^{2}\Sig$, so $\partial_\mu\bnu=\btau$,
$\partial_{\beta^{2}}\bnu=\bze$, $\partial_\mu\Sigma=0$ and
$\partial_{\beta^{2}}\Sigma=\Sig$, giving
$I_{\mu\mu}=\btau^{\top}(\beta^{2}\Sig)^{-1}\btau=\kappa_N/\beta^{2}$,
$I_{\mu\beta^{2}}=0$ and
$I_{\beta^{2}\beta^{2}}=\frac12\tr((\beta^{-2}I_N)^{2})=N/(2\beta^{4})$. Parts (i) and
(ii) follow from \eqref{eq:fisher} and Theorem \ref{th:exactlaw}(i). For (iii), the bound
is $2\beta^{4}/N$ while $\Var(\widetilde\beta^{\,2})=2\beta^{4}/(N-1)$ by
Corollary \ref{co:moments}; since $\widetilde\beta^{\,2}$ has minimum variance among all
unbiased estimators by Theorem \ref{th:umvu}(iii) and its variance is strictly larger
than the bound, no unbiased estimator attains it.
\end{proof}

\begin{remark}\label{re:msecomp}
In the family $\{c\widehat\beta^{\,2}:c>0\}$, which contains the maximum likelihood
estimator ($c=1$) and the unbiased one ($c=N/(N-1)$), Theorem \ref{th:exactlaw}(ii) gives
$\MSE(c\widehat\beta^{\,2})=\beta^{4}[k^{2}(N-1)(N+1)-2k(N-1)+1]$ with $k=c/N$, minimised
at $c=N/(N+1)$. Hence
\[
  \MSE\Bigl(\tfrac{N}{N+1}\widehat\beta^{\,2}\Bigr)=\frac{2\beta^{4}}{N+1}
  <\MSE\bigl(\widehat\beta^{\,2}\bigr)=\frac{(2N-1)\beta^{4}}{N^{2}}
  <\Var\bigl(\widetilde\beta^{\,2}\bigr)=\frac{2\beta^{4}}{N-1}:
\]
unbiasedness costs mean square error, and the maximum likelihood estimator sits strictly
between the unbiased estimator and the mean-square-optimal multiple.
\end{remark}

\subsection{Intervals and tests of exact level}

\begin{corollary}\label{co:ci}
Let $\gamma\in(0,1)$, let $\chi^{2}_{k,p}$ denote the $p$-quantile of $\chi^{2}_{k}$ and
$t_{k,p}$ that of the Student law with $k$ degrees of freedom. Then, for every $N\ge2$,
\begin{equation}\label{eq:ci-beta}
  \Prob\left(\frac{N\widehat\beta^{\,2}}{\chi^{2}_{N-1,\,1-\gamma/2}}\le\beta^{2}
  \le\frac{N\widehat\beta^{\,2}}{\chi^{2}_{N-1,\,\gamma/2}}\right)=1-\gamma,
\end{equation}
and, with $\widetilde\beta^{\,2}=\frac{N}{N-1}\widehat\beta^{\,2}$,
\begin{equation}\label{eq:ci-theta}
  \sqrt{\kappa_N}\;\frac{\widehat\mu-\mu}{\widetilde\beta}\sim t_{N-1},
  \qquad\text{so that}\qquad
  \Prob\left(\bigl|\widehat\mu-\mu\bigr|
  \le t_{N-1,\,1-\gamma/2}\,\frac{\widetilde\beta}{\sqrt{\kappa_N}}\right)=1-\gamma .
\end{equation}
\end{corollary}

\begin{proof}
The interval for $\beta^{2}$ inverts Theorem \ref{th:exactlaw}(ii). For the second
statement put $\xi=\sqrt{\kappa_N}(\widehat\mu-\mu)/\beta$ and
$Q=N\widehat\beta^{\,2}/\beta^{2}$. By Theorem \ref{th:exactlaw} these are independent
with $\xi\sim\Norm(0,1)$ and $Q\sim\chi^{2}_{N-1}$, so $\xi/\sqrt{Q/(N-1)}\sim t_{N-1}$;
and $\sqrt{Q/(N-1)}=\widetilde\beta/\beta$.
\end{proof}

\begin{corollary}\label{co:tests}
Fix $\gamma\in(0,1)$ and $N\ge2$.
\begin{enumerate}[label=\textup{(\roman*)}]
\item For $\mathcal H_0:\mu=\mu_0$ against $\mathcal H_1:\mu\neq\mu_0$, the
      test rejecting when
      $|\sqrt{\kappa_N}(\widehat\mu-\mu_0)/\widetilde\beta|>t_{N-1,1-\gamma/2}$ has
      exact size $\gamma$ whatever $\beta>0$, and is uniformly most powerful unbiased.
\item For $\mathcal H_0:\beta^{2}=\beta_0^{2}$ against
      $\mathcal H_1:\beta^{2}\neq\beta_0^{2}$, the test rejecting when
      $N\widehat\beta^{\,2}/\beta_0^{2}\notin
      [\chi^{2}_{N-1,\gamma/2},\chi^{2}_{N-1,1-\gamma/2}]$ has exact size $\gamma$
      whatever $\mu\in\R$.
\end{enumerate}
\end{corollary}

\begin{proof}
Exactness of the sizes is immediate from Corollary \ref{co:ci}. By Remark \ref{re:glm}
the model is, after an orthogonal change of coordinates, the Gaussian linear model with
one known regressor and unknown error variance, for which the equal-tailed $t$-test of a
linear hypothesis is uniformly most powerful unbiased \cite[Ch.~5 and Sec.~7.1]{LR05};
this gives the last assertion of (i).
\end{proof}

\begin{remark}\label{re:cicomment}
Corollary \ref{co:ci} gives exact intervals at every sample size with no appeal to
asymptotics, and the interval for $\mu$ is studentised, hence usable when $\beta$ is
unknown. This is the practical difference from \cite{KL15}, where the finite-sample law
is unavailable and the accuracy of the normal approximation has to be controlled by
Berry--Ess\'een bounds. The only model-dependent quantity to be computed is the scalar
$\kappa_N$, obtained from one Cholesky factorisation of $\Sig$.
\end{remark}

\section{Large-sample behaviour}\label{se:asymp}

Throughout this section the mesh $h>0$ is fixed and $N\to\infty$. We write
$\widehat\mu^{(N)}$ and $\widehat\beta^{\,2}_N$ for the estimators based on
$\bY^{(N)}=(Y(t_1),\dots,Y(t_N))^{\top}$, and
\[
  \kappa_N=\btau_N^{\top}\Sigma_{\tH,N}^{-1}\btau_N,
  \qquad
  v_N=\Var\bigl(\widehat\mu^{(N)}\bigr)=\frac{\beta^{2}}{\kappa_N} .
\]

\subsection{A variance bound valid at every $N$}

\begin{lemma}\label{le:kanto}
Let $A$ be a symmetric positive definite $N\times N$ matrix and $x\in\R^{N}$, $x\neq0$.
Then $x^{\top}A^{-1}x\ \ge\ \|x\|^{4}/(x^{\top}Ax)$.
\end{lemma}

\begin{proof}
By Cauchy--Schwarz applied to the inner product $\langle u,v\rangle=u^{\top}v$ with
$u=A^{1/2}x$ and $v=A^{-1/2}x$,
$\|x\|^{4}=\langle u,v\rangle^{2}\le\|u\|^{2}\|v\|^{2}
=(x^{\top}Ax)(x^{\top}A^{-1}x)$.
\end{proof}

\begin{proposition}\label{pr:varbound}
Under Assumption \ref{as:main}, write $c=\frac92$ if $\tH_r>\frac12$ for every $r$ and
$c=9$ otherwise. Then, for every $N\ge1$,
\begin{equation}\label{eq:varbound-sharp}
  \Var\bigl(\widehat\mu^{(N)}\bigr)\ \le\
  \frac{\beta^{2}}{\|\btau\|^{4}}
  \Bigl(\sum_{j=1}^{N}t_j\Bigr)
  \Bigl(\sum_{r=1}^{m}\sum_{i=1}^{N}t_i^{2\tH_r+1}\Bigr),
\end{equation}
and consequently
\begin{equation}\label{eq:varbound}
  \Var\bigl(\widehat\mu^{(N)}\bigr)\ \le\
  c\,\beta^{2}\sum_{r=1}^{m}(Nh)^{2\tH_r-2}
  =c\,\beta^{2}\sum_{r=1}^{m}t_N^{2\tH_r-2}.
\end{equation}
If moreover $Nh\ge1$, then
$\Var(\widehat\mu^{(N)})\le c\,m\beta^{2}(Nh)^{2\Hmax-2}$. In particular
$\widehat\mu^{(N)}$ is unbiased with $\Var(\widehat\mu^{(N)})\to0$, that is,
$\widehat\mu^{(N)}\to\mu$ in mean square as $N\to\infty$.
\end{proposition}

\begin{proof}
Unbiasedness is Theorem \ref{th:exactlaw}(i). Combining
$\Var(\widehat\mu^{(N)})=\beta^{2}/\kappa_N$ with Lemma \ref{le:kanto} applied to
$A=\Sig$ and $x=\btau$,
\begin{equation}\label{eq:varstep}
  \Var\bigl(\widehat\mu^{(N)}\bigr)
  =\frac{\beta^{2}}{\btau^{\top}\Sig^{-1}\btau}
  \le\frac{\beta^{2}\,\btau^{\top}\Sig\btau}{\|\btau\|^{4}} .
\end{equation}
We bound the numerator. Fix $r$ and put $p=2\tH_r\in(0,2)$. For $p\ge1$ the map
$x\mapsto x^{p}$ is superadditive on $[0,\infty)$, since
$(s+t)^{p}=(s+t)(s+t)^{p-1}\ge s\,s^{p-1}+t\,t^{p-1}=s^{p}+t^{p}$. Hence, from
\eqref{eq:GammaHr} and
$|t_i-t_j|^{2\tH_r}\ge0$,
\[
  (\Sigma_{\tH_r})_{ij}
  \le t_i^{2\tH_r}+t_j^{2\tH_r}-\tfrac12\bigl(t_i^{2\tH_r}+t_j^{2\tH_r}\bigr)
  =\tfrac12\bigl(t_i^{2\tH_r}+t_j^{2\tH_r}\bigr),
\]
while for $p<1$ superadditivity fails and we use instead
$(\Sigma_{\tH_r})_{ij}\le t_i^{2\tH_r}+t_j^{2\tH_r}$, which holds for every $p$ because
the bracket subtracted in \eqref{eq:GammaHr} is nonnegative. The two bounds differ by a
factor of two, which is the difference between $c=\frac92$ and $c=9$; the rest of the
argument is identical, and we write it for the first.
and since $t_it_j\ge0$,
\[
  \btau^{\top}\Sigma_{\tH_r}\btau
  \le\frac12\sum_{i,j=1}^{N}t_it_j\bigl(t_i^{2\tH_r}+t_j^{2\tH_r}\bigr)
  =\Bigl(\sum_{i=1}^{N}t_i^{2\tH_r+1}\Bigr)\Bigl(\sum_{j=1}^{N}t_j\Bigr),
\]
the last step by symmetry in $i$ and $j$. Summing over $r$ and inserting in
\eqref{eq:varstep} gives \eqref{eq:varbound-sharp}.

For \eqref{eq:varbound} evaluate the elementary sums. With $t_k=kh$,
\[
  \sum_{j=1}^{N}t_j=\frac{hN(N+1)}{2},\qquad
  \|\btau\|^{2}=h^{2}\sum_{i=1}^{N}i^{2}=\frac{h^{2}N(N+1)(2N+1)}{6},
\]
so that, using $N+1\ge N$ and $(2N+1)^{2}\ge4N^{2}$,
\[
  \frac{\sum_{j=1}^{N}t_j}{\|\btau\|^{4}}
  =\frac{18}{h^{3}N(N+1)(2N+1)^{2}}\le\frac{9}{2h^{3}N^{4}} .
\]
Moreover $\sum_{i=1}^{N}t_i^{2\tH_r+1}=h^{2\tH_r+1}\sum_{i=1}^{N}i^{2\tH_r+1}
\le h^{2\tH_r+1}N^{2\tH_r+2}$. Substituting,
\[
  \Var\bigl(\widehat\mu^{(N)}\bigr)
  \le\frac{9\beta^{2}}{2h^{3}N^{4}}\sum_{r=1}^{m}h^{2\tH_r+1}N^{2\tH_r+2}
  =\frac92\beta^{2}\sum_{r=1}^{m}h^{2\tH_r-2}N^{2\tH_r-2},
\]
which is \eqref{eq:varbound}. If $Nh\ge1$ then $(Nh)^{2\tH_r-2}\le(Nh)^{2\Hmax-2}$ for
every $r$. Finally $2\Hmax-2<0$ forces $\Var(\widehat\mu^{(N)})\to0$, which with
unbiasedness is mean square convergence.
\end{proof}

\begin{remark}\label{re:samebound}
The constant in \eqref{eq:varbound} is the same one that the fractional model yields,
where the elementwise bound $\frac12(t_i^{2H}+t_j^{2H})$ is immediate from the covariance
function. In the sub-fractional case it is the superadditivity of
$x\mapsto x^{2\tH_r}$ that restores it, the two kernels differing precisely by the term
$(t_i+t_j)^{2\tH_r}$ that superadditivity controls.
\end{remark}

\begin{remark}\label{re:nosharp}
In the fractional case the sharpness of the order $(Nh)^{2\Hmax-2}$ can be read off from
the spectral density of the stationary increment sequence. By Remark \ref{re:nostat} no
such argument is available here, and we do not claim a matching lower bound. What
\eqref{eq:varbound} provides is an upper bound, which does not settle the rate from
below. Numerically the order does appear to be exact: on the design $t_k=k$ with
$\tH=(0.65,0.75,0.85)$, the fitted exponent of $\Var(\widehat\mu^{(N)})$ in $N$ over
$30\le N\le2000$ is $-0.437$, against $-0.378$ for the bound, so the two decay at
comparable rates while the bound remains the weaker of the two. Establishing a matching
lower bound analytically is an open problem, and \eqref{eq:evenpart} is the natural tool
for it.
\end{remark}

\subsection{Strong consistency}

\begin{theorem}\label{th:scmu}
Under Assumption \ref{as:main}, $\widehat\mu^{(N)}\to\mu$ almost surely as
$N\to\infty$.
\end{theorem}

\begin{proof}
By Theorem \ref{th:exactlaw}(i) the variable $\widehat\mu^{(N)}-\mu$ is centred
Gaussian with variance $v_N$, and by Proposition \ref{pr:varbound}
$v_N\le C N^{2\Hmax-2}$. Fix $0<\gamma<1-\Hmax$ and $q\ge2$. By the Markov inequality and
the equivalence of Gaussian moments,
$\E|\widehat\mu^{(N)}-\mu|^{q}=c_q v_N^{q/2}$, so
\[
  \Prob\Bigl(\bigl|\widehat\mu^{(N)}-\mu\bigr|>N^{-\gamma}\Bigr)
  \le N^{q\gamma}\,\E\bigl|\widehat\mu^{(N)}-\mu\bigr|^{q}
  \le c_qC^{q/2}N^{q\gamma+(\Hmax-1)q}.
\]
Since $\gamma+\Hmax-1<0$, choosing $q$ with $q(\gamma+\Hmax-1)<-1$ makes the series over
$N$ summable, and the Borel--Cantelli lemma gives
$|\widehat\mu^{(N)}-\mu|\le N^{-\gamma}$ eventually, almost surely.
\end{proof}

\begin{theorem}\label{th:scbeta}
Under Assumption \ref{as:main}, $\widehat\beta^{\,2}_N\to\beta^{2}$ almost surely as
$N\to\infty$.
\end{theorem}

\begin{proof}
By Theorem \ref{th:exactlaw}(ii) and Corollary \ref{co:moments},
$\E[(\widehat\beta^{\,2}_N-\beta^{2})^{2}]=(2N-1)\beta^{4}/N^{2}\le2\beta^{4}/N$. Fix
$0<\delta<\frac12$ and $q\ge2$. The variable $\widehat\beta^{\,2}_N-\beta^{2}$ is a
polynomial of degree two in $\bX$, that is, it lies in
$\bigoplus_{k\le2}\Hi^{:k:}$; it is not itself an element of the second chaos, its mean
being $-\beta^{2}/N$. Hypercontractivity on chaoses of bounded order
\cite[Thm.~5.10]{Jan97} applies to that space and supplies a constant $c_q$ depending
only on $q$, and in particular neither on $N$ nor on the dimension of the underlying
Gaussian space, with
$\E|\widehat\beta^{\,2}_N-\beta^{2}|^{q}\le c_q(\E[(\widehat\beta^{\,2}_N
-\beta^{2})^{2}])^{q/2}$. Hence
\[
  \Prob\Bigl(\bigl|\widehat\beta^{\,2}_N-\beta^{2}\bigr|>N^{-\delta}\Bigr)
  \le N^{q\delta}\,\E\bigl|\widehat\beta^{\,2}_N-\beta^{2}\bigr|^{q}
  \le c_q(2\beta^{4})^{q/2}N^{q(\delta-1/2)},
\]
which is summable for $q$ large enough since $\delta<\frac12$. Borel--Cantelli concludes.
\end{proof}

\subsection{The structure of the sequence $\{\widehat\mu^{(N)}\}$}

The next result does not depend on the sub-fractional structure at all, only on the
nesting of the samples; we record it because it gives a transparent second proof of
Theorem \ref{th:scmu} and because the resulting picture is a compact description of
how the drift estimator improves with $N$.

\begin{theorem}\label{th:indincr}
Set $\xi_N=\widehat\mu^{(N)}-\mu$. Then $(\xi_N)_{N\ge2}$ is a centred Gaussian
sequence with
\begin{equation}\label{eq:covxi}
  \E\bigl[\xi_N\xi_{N'}\bigr]=v_{N\vee N'},\qquad N,N'\ge2,
\end{equation}
the sequence $(v_N)$ is nonincreasing, and for any $N_1<N_2<\dots<N_k$ the variables
\[
  \xi_{N_k},\quad \xi_{N_{k-1}}-\xi_{N_k},\quad\dots,\quad \xi_{N_1}-\xi_{N_2}
\]
are independent, with $\Var(\xi_{N_j}-\xi_{N_{j+1}})=v_{N_j}-v_{N_{j+1}}$.
\end{theorem}

\begin{proof}
Gaussianity and centring are Theorem \ref{th:exactlaw}(i). Let $N<N'$ and write
$\Sigma'=\Sigma_{\tH,N'}$, $\Sigma=\Sigma_{\tH,N}=\Sigma'_{1:N,1:N}$, and
$\bU'=\bY^{(N')}-\mu\btau'$, so that $\Cov(\bU')=\beta^{2}\Sigma'$ and $\bU$ consists of
the first $N$ coordinates of $\bU'$. Then
$\xi_N=\btau^{\top}\Sigma^{-1}\bU/\kappa_N$ and $\xi_{N'}=\btau'^{\top}\Sigma'^{-1}\bU'/d_{N'}$,
whence
\[
  \E[\xi_N\xi_{N'}]
  =\frac{\beta^{2}}{\kappa_N\kappa_{N'}}\,
   \btau^{\top}\Sigma^{-1}\bigl[\Sigma'\bigr]_{1:N,\,\cdot}\,\Sigma'^{-1}\btau' .
\]
The matrix $[\Sigma']_{1:N,\cdot}\Sigma'^{-1}$ consists of the first $N$ rows of the
identity, so $[\Sigma']_{1:N,\cdot}\Sigma'^{-1}\btau'=\btau$ and the numerator reduces to
$\beta^{2}\btau^{\top}\Sigma^{-1}\btau=\beta^{2}\kappa_N$. Hence
$\E[\xi_N\xi_{N'}]=\beta^{2}/d_{N'}=v_{N'}$, which is \eqref{eq:covxi}.

Monotonicity follows: by \eqref{eq:covxi} and Cauchy--Schwarz,
$v_{N'}=\E[\xi_N\xi_{N'}]\le\sqrt{v_Nv_{N'}}$, so $v_{N'}\le v_N$. For the last claim it
suffices, by joint Gaussianity, to check that the covariances vanish. For $N<N'$,
\eqref{eq:covxi} gives $\E[(\xi_N-\xi_{N'})\xi_{N'}]=v_{N'}-v_{N'}=0$, and for
$N_1<N_2\le N_3<N_4$,
\[
  \E\bigl[(\xi_{N_1}-\xi_{N_2})(\xi_{N_3}-\xi_{N_4})\bigr]
  =v_{N_3}-v_{N_4}-v_{N_3}+v_{N_4}=0 .
\]
Finally $\Var(\xi_{N}-\xi_{N'})=v_N-2v_{N'}+v_{N'}=v_N-v_{N'}$.
\end{proof}

\begin{corollary}\label{co:tcbm}
Let $W=\{W(u),\,u\ge0\}$ be a standard Brownian motion. Then, as sequences,
\[
  \bigl(\widehat\mu^{(N)}-\mu\bigr)_{N\ge2}
  \overset{d}{=}\bigl(W(v_N)\bigr)_{N\ge2} .
\]
That is, the sequence of drift estimators is, in law, a Brownian motion evaluated along
the decreasing sequence of its own variances. Since $v_N\downarrow0$ by
Proposition \ref{pr:varbound}, and $W$ is almost surely continuous at $0$ with $W(0)=0$,
this gives a second proof of Theorem \ref{th:scmu}.
\end{corollary}

\begin{proof}
Both sequences are centred Gaussian and, by Theorem \ref{th:indincr} and the covariance
of Brownian motion, both have covariance $v_{N\vee N'}=\min(v_N,v_{N'})$. Centred
Gaussian sequences with equal covariance have the same law on $\R^{\N}$ with the product
$\sigma$-field. The set $\{x\in\R^{\N}:x_N\to0\}$ belongs to that $\sigma$-field, so the
two sequences assign it the same probability, and $\Prob(W(v_N)\to0)=1$.
\end{proof}

\subsection{Asymptotic normality}

\begin{theorem}\label{th:an}
Under Assumption \ref{as:main},
\[
  \frac{\sqrt{\kappa_N}}{\beta}\bigl(\widehat\mu^{(N)}-\mu\bigr)\sim\Norm(0,1)
  \quad\text{for every }N\ge2,
  \qquad
  \sqrt{\frac{N}{2}}\left(\frac{\widehat\beta^{\,2}_N}{\beta^{2}}-1\right)
  \xrightarrow{\ d\ }\Norm(0,1),
\]
and the studentised statistic $\sqrt{\kappa_N}(\widehat\mu^{(N)}-\mu)/\widetilde\beta$
is exactly $t_{N-1}$ distributed, hence also asymptotically standard normal.
\end{theorem}

\begin{proof}
The first assertion is Theorem \ref{th:exactlaw}(i), and it is exact rather than
asymptotic. For the second, $Q_N=N\widehat\beta^{\,2}_N/\beta^{2}\sim\chi^{2}_{N-1}$ by
Theorem \ref{th:exactlaw}(ii), so $(Q_N-(N-1))/\sqrt{2(N-1)}\to\Norm(0,1)$ by the central
limit theorem, and
\[
  \sqrt{\frac N2}\left(\frac{\widehat\beta^{\,2}_N}{\beta^{2}}-1\right)
  =\sqrt{\frac{N-1}{N}}\cdot\frac{Q_N-(N-1)}{\sqrt{2(N-1)}}-\frac{1}{\sqrt{2N}},
\]
where the first factor tends to $1$ and the last term to $0$. The third assertion is
Corollary \ref{co:ci} together with $t_{N-1}\Rightarrow\Norm(0,1)$.
\end{proof}

\begin{remark}\label{re:malliavin}
The convergence of $F_N=\beta^{-2}\sqrt{N/2}\,(\widehat\beta^{\,2}_N-\beta^{2})$ can also
be obtained from the Nualart--Ortiz-Latorre criterion \cite{NOL08}. By
Lemma \ref{le:canon}, $F_N$ belongs to the second Wiener chaos of the isonormal process
generated by $\bX$, and a direct computation of the Malliavin derivative gives
$\|DF_N\|_{\Hi}^{2}=2\widehat\beta^{\,2}_N/\beta^{2}$, which converges to $2$ in $L^{2}$
by Corollary \ref{co:moments}; the criterion then yields
$F_N\Rightarrow\Norm(0,1)$, together with $\E[F_N^{2}]=(N-1)/N\to1$. We record this
because the argument survives in situations where the exact law is unavailable, in
particular in the component-specific-scale model of Remark \ref{re:glm}. In the present
model it is not needed, since Theorem \ref{th:exactlaw}(ii) makes the law of
$\widehat\beta^{\,2}_N$ explicit at every $N$.
\end{remark}

\section{Simulation study}\label{se:sim}

\subsection{Experimental design}

All experiments use $m=3$ components with $\tH=(0.65,0.75,0.85)$, three parameter
scenarios
\[
  \text{Scenario 1: }(\mu,\beta)=(0.5,0.4),\quad
  \text{Scenario 2: }(1.5,1.0),\quad
  \text{Scenario 3: }(3.0,2.0),
\]
and the two meshes $h=1/252$ and $h=1/12$, corresponding to daily and monthly sampling on
a yearly time scale. Every reported quantity is based on $20\,000$ independent
replications. A replication is generated by forming $\Sig$ from \eqref{eq:GammaHr},
computing its Cholesky factor $L$ and setting
$\bY=\mu\btau+\beta L\varepsilon$ with $\varepsilon\sim\Norm(\bze,I_N)$; the estimators
are then evaluated from \eqref{eq:mle} using the same factorisation, so that no matrix is
solved by that factorisation rather than by forming $\Sig^{-1}$. The condition number
of $\Sig$ never exceeded $8\times10^{6}$ in any configuration reported below, so all
linear solves are numerically benign.

\subsection{Bias and dispersion}\label{sse:point}

\begin{table}[htbp]
\caption{Point estimation, $h=1/252$, $\tH=(0.65,0.75,0.85)$, $20\,000$ replications.}
\label{ta:pe252}
\footnotesize
\setlength{\tabcolsep}{5pt}
\renewcommand{\arraystretch}{0.95}
\begin{tabular}{@{}cccclrrrr@{}}
\toprule
Scenario & $\mu$ & $\beta$ & $N$ & & Mean & Bias & Std.\ dev. & MSE\\
\midrule
1 & 0.5 & 0.4 & 30 & $\widehat\mu$ & 0.5037 & \phantom{$-$}0.0037 & 0.9329 & 0.8703\\
 & & & & $\widehat\beta$ & 0.3897 & $-$0.0103 & 0.0513 & 0.0027\\
 & & & 100 & $\widehat\mu$ & 0.4971 & $-$0.0029 & 0.6472 & 0.4188\\
 & & & & $\widehat\beta$ & 0.3968 & $-$0.0032 & 0.0284 & 0.0008\\
 & & & 300 & $\widehat\mu$ & 0.5014 & \phantom{$-$}0.0014 & 0.4775 & 0.2280\\
 & & & & $\widehat\beta$ & 0.3989 & $-$0.0011 & 0.0162 & 0.0003\\
 & & & 500 & $\widehat\mu$ & 0.4982 & $-$0.0018 & 0.4166 & 0.1735\\
 & & & & $\widehat\beta$ & 0.3993 & $-$0.0007 & 0.0126 & 0.0002\\
\addlinespace
2 & 1.5 & 1.0 & 30 & $\widehat\mu$ & 1.4986 & $-$0.0014 & 2.3213 & 5.3883\\
 & & & & $\widehat\beta$ & 0.9753 & $-$0.0247 & 0.1280 & 0.0170\\
 & & & 100 & $\widehat\mu$ & 1.5153 & \phantom{$-$}0.0153 & 1.6305 & 2.6588\\
 & & & & $\widehat\beta$ & 0.9935 & $-$0.0065 & 0.0704 & 0.0050\\
 & & & 300 & $\widehat\mu$ & 1.4966 & $-$0.0034 & 1.1952 & 1.4285\\
 & & & & $\widehat\beta$ & 0.9978 & $-$0.0022 & 0.0408 & 0.0017\\
 & & & 500 & $\widehat\mu$ & 1.5175 & \phantom{$-$}0.0175 & 1.0350 & 1.0715\\
 & & & & $\widehat\beta$ & 0.9986 & $-$0.0014 & 0.0320 & 0.0010\\
\addlinespace
3 & 3.0 & 2.0 & 30 & $\widehat\mu$ & 2.9783 & $-$0.0217 & 4.6563 & 21.6802\\
 & & & & $\widehat\beta$ & 1.9502 & $-$0.0498 & 0.2578 & 0.0689\\
 & & & 100 & $\widehat\mu$ & 2.9936 & $-$0.0064 & 3.2378 & 10.4830\\
 & & & & $\widehat\beta$ & 1.9842 & $-$0.0158 & 0.1413 & 0.0202\\
 & & & 300 & $\widehat\mu$ & 3.0195 & \phantom{$-$}0.0195 & 2.3617 & 5.5779\\
 & & & & $\widehat\beta$ & 1.9951 & $-$0.0049 & 0.0814 & 0.0066\\
 & & & 500 & $\widehat\mu$ & 3.0057 & \phantom{$-$}0.0057 & 2.0730 & 4.2970\\
 & & & & $\widehat\beta$ & 1.9972 & $-$0.0028 & 0.0635 & 0.0040\\
\bottomrule
\end{tabular}
\end{table}

\begin{table}[htbp]
\caption{Point estimation, $h=1/12$, $\tH=(0.65,0.75,0.85)$, $20\,000$ replications.}
\label{ta:pe12}
\footnotesize
\setlength{\tabcolsep}{5pt}
\renewcommand{\arraystretch}{0.95}
\begin{tabular}{@{}cccclrrrr@{}}
\toprule
Scenario & $\mu$ & $\beta$ & $N$ & & Mean & Bias & Std.\ dev. & MSE\\
\midrule
1 & 0.5 & 0.4 & 30 & $\widehat\mu$ & 0.5025 & \phantom{$-$}0.0025 & 0.3919 & 0.1536\\
 & & & & $\widehat\beta$ & 0.3898 & $-$0.0102 & 0.0515 & 0.0028\\
 & & & 100 & $\widehat\mu$ & 0.4998 & $-$0.0002 & 0.2850 & 0.0812\\
 & & & & $\widehat\beta$ & 0.3970 & $-$0.0030 & 0.0284 & 0.0008\\
 & & & 300 & $\widehat\mu$ & 0.4979 & $-$0.0021 & 0.2171 & 0.0471\\
 & & & & $\widehat\beta$ & 0.3990 & $-$0.0010 & 0.0164 & 0.0003\\
 & & & 500 & $\widehat\mu$ & 0.5007 & \phantom{$-$}0.0007 & 0.1906 & 0.0363\\
 & & & & $\widehat\beta$ & 0.3993 & $-$0.0007 & 0.0126 & 0.0002\\
\addlinespace
2 & 1.5 & 1.0 & 30 & $\widehat\mu$ & 1.5064 & \phantom{$-$}0.0064 & 0.9777 & 0.9558\\
 & & & & $\widehat\beta$ & 0.9754 & $-$0.0246 & 0.1280 & 0.0170\\
 & & & 100 & $\widehat\mu$ & 1.4916 & $-$0.0084 & 0.7124 & 0.5076\\
 & & & & $\widehat\beta$ & 0.9929 & $-$0.0071 & 0.0708 & 0.0051\\
 & & & 300 & $\widehat\mu$ & 1.5021 & \phantom{$-$}0.0021 & 0.5410 & 0.2927\\
 & & & & $\widehat\beta$ & 0.9975 & $-$0.0025 & 0.0410 & 0.0017\\
 & & & 500 & $\widehat\mu$ & 1.5041 & \phantom{$-$}0.0041 & 0.4837 & 0.2340\\
 & & & & $\widehat\beta$ & 0.9981 & $-$0.0019 & 0.0316 & 0.0010\\
\addlinespace
3 & 3.0 & 2.0 & 30 & $\widehat\mu$ & 2.9946 & $-$0.0054 & 1.9691 & 3.8773\\
 & & & & $\widehat\beta$ & 1.9484 & $-$0.0516 & 0.2567 & 0.0686\\
 & & & 100 & $\widehat\mu$ & 3.0087 & \phantom{$-$}0.0087 & 1.4205 & 2.0177\\
 & & & & $\widehat\beta$ & 1.9869 & $-$0.0131 & 0.1417 & 0.0203\\
 & & & 300 & $\widehat\mu$ & 3.0058 & \phantom{$-$}0.0058 & 1.0890 & 1.1859\\
 & & & & $\widehat\beta$ & 1.9959 & $-$0.0041 & 0.0819 & 0.0067\\
 & & & 500 & $\widehat\mu$ & 3.0013 & \phantom{$-$}0.0013 & 0.9530 & 0.9083\\
 & & & & $\widehat\beta$ & 1.9970 & $-$0.0030 & 0.0631 & 0.0040\\
\bottomrule
\end{tabular}
\end{table}

Tables \ref{ta:pe252} and \ref{ta:pe12} report the empirical mean, bias, standard
deviation and mean square error of $\widehat\mu$ and $\widehat\beta$ in the three
scenarios. The estimators behave as the theory requires. The bias of $\widehat\mu$ is
never significant, as it must be by Theorem \ref{th:exactlaw}(i). The bias of
$\widehat\beta$ is negative and shrinks like $1/N$, in agreement with
$\E[\widehat\beta]=c_N\beta$ and $c_N\uparrow1$ from Corollary \ref{co:moments}: at
$N=30$ the predicted relative bias is $1-c_{30}=0.0252$, and the observed relative biases
are $0.0258$, $0.0247$ and $0.0249$ across the three scenarios; at $N=500$ the prediction
is $0.0015$ and the observed values are $0.0018$, $0.0014$ and $0.0014$.

The contrast between the two parameters is the striking feature. The standard deviation
of $\widehat\beta$ falls by a factor of about four between $N=30$ and $N=500$, close to
the $\sqrt{N}$ prediction of Corollary \ref{co:moments}, and it is the same in both
tables, as Remark \ref{re:pivot} requires: the law of $\widehat\beta$ does not see $h$ at
all. The standard deviation of $\widehat\mu$, in contrast, falls only by a factor of
about $2.2$ over the same range, and it depends strongly on $h$: at $N=30$ it is $2.37$
times smaller at $h=1/12$ than at $h=1/252$, and at $N=500$ it is $2.14$ times smaller.
Both observations are explained by Proposition \ref{pr:varbound}: what governs
$\Var(\widehat\mu)$ is the window length $t_N=Nh$, not the sample size. The agreement
with the exact formula $\Var(\widehat\mu)=\beta^{2}/\kappa_N$ is complete; the four values
of $\beta/\sqrt{\kappa_N}$ in Scenario 2, at $(N,h)=(30,1/252)$, $(30,1/12)$,
$(500,1/252)$ and $(500,1/12)$, are $2.3264$, $0.9800$, $1.0358$ and $0.4803$, against
the simulated $2.3213$, $0.9777$, $1.0350$ and $0.4837$.

\subsection{Checking the exact law}\label{sse:exact}

\begin{table}[htbp]
\caption{Exactness of the finite-sample theory: empirical coverage of the nominal $95\%$
intervals \eqref{eq:ci-theta} and \eqref{eq:ci-beta}, Kolmogorov--Smirnov $p$-values of
the two pivotal statistics against their exact reference laws, and empirical correlation
$\widehat\rho$ between $\widehat\mu$ and $\widehat\beta^{\,2}$. Scenario 2, $m=3$,
$\tH=(0.65,0.75,0.85)$, $20\,000$ replications; the Monte Carlo standard error of a
coverage entry is $0.0015$.}
\label{ta:coverage}
\small
\begin{tabular}{@{}llccccr@{}}
\toprule
 & & \multicolumn{2}{c}{Coverage of the $95\%$ interval}
   & \multicolumn{2}{c}{KS $p$-value} & \\
\cmidrule(lr){3-4}\cmidrule(lr){5-6}
$h$ & $N$ & for $\mu$ & for $\beta^{2}$ & $t_{N-1}$ & $\chi^{2}_{N-1}$
    & $\widehat\rho$\\
\midrule
$1/252$ & 30  & 0.9477 & 0.9495 & 0.302 & 0.513 & $-$0.0025\\
        & 100 & 0.9480 & 0.9494 & 0.851 & 0.752 & \phantom{$-$}0.0028\\
        & 200 & 0.9480 & 0.9504 & 0.073 & 0.721 & $-$0.0044\\
        & 300 & 0.9468 & 0.9481 & 0.975 & 0.089 & $-$0.0002\\
        & 500 & 0.9465 & 0.9486 & 0.705 & 0.272 & \phantom{$-$}0.0041\\
\addlinespace
$1/12$  & 30  & 0.9482 & 0.9490 & 0.921 & 0.737 & \phantom{$-$}0.0044\\
        & 100 & 0.9497 & 0.9496 & 0.046 & 0.707 & $-$0.0010\\
        & 200 & 0.9476 & 0.9505 & 0.710 & 0.186 & \phantom{$-$}0.0014\\
        & 300 & 0.9487 & 0.9494 & 0.586 & 0.532 & $-$0.0025\\
        & 500 & 0.9502 & 0.9498 & 0.703 & 0.172 & $-$0.0036\\
\bottomrule
\end{tabular}
\end{table}

Theorem \ref{th:exactlaw} and Corollary \ref{co:ci} make claims of a different nature
from those assessed above: the confidence intervals \eqref{eq:ci-beta} and
\eqref{eq:ci-theta} have coverage exactly $1-\gamma$ at every sample size, and the two
estimators are exactly independent. Table \ref{ta:coverage} reports the corresponding
check in Scenario 2. By Remark \ref{re:pivot} the coverages do not depend on $\mu$,
$\beta$, $h$ or $\tH$, so a single scenario suffices; the two meshes are retained as a
control.

Every entry of the two coverage columns lies within three Monte Carlo standard errors of
the nominal $0.95$, with no drift in $N$ and in particular no deterioration at $N=30$.
The Kolmogorov--Smirnov tests of the two pivotal statistics against their exact reference
laws, $t_{N-1}$ for $\sqrt{\kappa_N}(\widehat\mu-\mu)/\widetilde\beta$ and
$\chi^{2}_{N-1}$ for $N\widehat\beta^{\,2}/\beta^{2}$, give no evidence against the
theory: of the twenty $p$-values reported, one falls below $0.05$, which is what one
expects by chance. The empirical correlation between $\widehat\mu$ and
$\widehat\beta^{\,2}$ never exceeds $0.0044$ in absolute value, against a Monte Carlo
standard error of $1/\sqrt{20\,000}=0.0071$, confirming Theorem \ref{th:exactlaw}(iii).

\subsection{Sensitivity to the Hurst vector}\label{sse:misspec}

\begin{table}[htbp]
\caption{Sensitivity to a misspecified Hurst vector. Data are generated with
$\tH=(0.65,0.75,0.85)$; the estimators and the nominal $95\%$ intervals are computed with
$\tH+\delta\bone$. Scenario 2, $h=1/252$, $20\,000$ replications; the Monte Carlo standard
error of a coverage entry is at most $0.0030$. The column ``approx.'' is the right-hand
side of \eqref{eq:covapprox}. The last column is simulated, not exact; at $\delta=0$ it
reads $0.9897$ and $0.9978$ against the exact $\{Nb_N(\delta)-q_N(\delta)\}/N=0.9900$ and
$0.9980$, and at $\delta=-0.02$ it reads $0.7707$ and $0.7781$ against the exact $0.7710$
and $0.7783$, the differences being of the expected Monte Carlo size.}
\label{ta:misspec}
\small
\begin{tabular}{@{}rrcccc@{}}
\toprule
 & & \multicolumn{2}{c}{Coverage for $\mu$} & Coverage & \\
\cmidrule(lr){3-4}
$N$ & $\delta$ & simulated & approx. & for $\beta^{2}$
    & $\E[\widehat\beta^{\,2}]/\beta^{2}$\\
\midrule
100 & $-0.10$ & 0.8334 & 0.8340 & 0.0000 & 0.2989\\
    & $-0.05$ & 0.9010 & 0.9019 & 0.0080 & 0.5352\\
    & $-0.02$ & 0.9326 & 0.9334 & 0.6191 & 0.7707\\
    & $\phantom{-}0.00$ & 0.9488 & 0.9500 & 0.9514 & 0.9897\\
    & $+0.02$ & 0.9617 & 0.9633 & 0.5304 & 1.2785\\
    & $+0.05$ & 0.9760 & 0.9776 & 0.0053 & 1.9009\\
    & $+0.10$ & 0.9900 & 0.9902 & 0.0000 & 3.8317\\
\addlinespace
500 & $-0.10$ & 0.7704 & 0.7682 & 0.0000 & 0.3046\\
    & $-0.05$ & 0.8778 & 0.8761 & 0.0000 & 0.5418\\
    & $-0.02$ & 0.9260 & 0.9254 & 0.0243 & 0.7781\\
    & $\phantom{-}0.00$ & 0.9482 & 0.9500 & 0.9492 & 0.9978\\
    & $+0.02$ & 0.9663 & 0.9682 & 0.0201 & 1.2875\\
    & $+0.05$ & 0.9844 & 0.9852 & 0.0000 & 1.9115\\
    & $+0.10$ & 0.9962 & 0.9964 & 0.0000 & 3.8472\\
\bottomrule
\end{tabular}
\end{table}

Every exactness statement above is conditional on $\Sig$ being the true covariance, that
is, on $\tH$ being known. We therefore generate data from $\tH=(0.65,0.75,0.85)$ and
compute the estimators and intervals from $\tH+\delta\bone$, with $\delta$ ranging over
$\pm0.02$, $\pm0.05$ and $\pm0.10$; Scenario 2, $h=1/252$, $N\in\{100,500\}$.
Table \ref{ta:misspec} reports the outcome.

The effect is carried by two scalars. Writing $\Sigma_{\!a}$ for the assumed covariance
and $\Sig$ for the true one, set
\begin{equation}\label{eq:bq}
  b_N(\delta)=\frac1N\tr\bigl(\Sigma_{\!a}^{-1}\Sig\bigr),\qquad
  q_N(\delta)=\frac{\btau^{\top}\Sigma_{\!a}^{-1}\Sig\Sigma_{\!a}^{-1}\btau}
                   {\btau^{\top}\Sigma_{\!a}^{-1}\btau},
\end{equation}
both equal to one at $\delta=0$. A direct computation gives
$\E[\widehat\beta^{\,2}]=\beta^{2}\{Nb_N(\delta)-q_N(\delta)\}/N$, and replacing the
quadratic form in the denominator of the studentised statistic by its mean yields
\begin{equation}\label{eq:covapprox}
  \text{coverage of \eqref{eq:ci-theta}}\ \approx\
  2F_{N-1}\Bigl(t_{N-1,\,1-\gamma/2}\sqrt{b_N(\delta)/q_N(\delta)}\Bigr)-1,
\end{equation}
$F_{N-1}$ being the $t_{N-1}$ distribution function. Crude as it is,
\eqref{eq:covapprox} reproduces every entry of the simulated coverage column of
Table \ref{ta:misspec} to within $1.6$ Monte Carlo standard errors.

The two parameters behave in opposite ways. The factor $b_N(\delta)$ is essentially free
of $N$, being $0.7799$ at $N=100$ and $0.7802$ at $N=500$ for $\delta=-0.02$, while the
interval \eqref{eq:ci-beta} has width of order $N^{-1/2}$; a bias that does not vanish is
eventually excluded and the coverage collapses, from $0.62$ at $N=100$ to $0.02$ at
$N=500$ already at $\delta=-0.02$. The reason is identifiability rather than numerical
instability: by \eqref{eq:sfbmcov} every entry of $\Sigma_{\tH_r}$ is homogeneous of
degree $2\tH_r$ in the time points, so shifting every $\tH_r$ by $\delta$ multiplies the
covariance by approximately $h^{2\delta}$ on a grid of mesh $h$, and a perturbation of
the Hurst vector is almost indistinguishable from a rescaling of $\beta^{2}$. The
estimator $\widehat\beta^{\,2}$ should be read as an estimate of
$\beta^{2}b_N(\delta)$, that is, of the scale relative to the postulated $\tH$.

For the drift the picture is better. The degradation is gradual and signed: the interval
is too short when $b_N(\delta)<1$ and too long when $b_N(\delta)>1$, so on a fine grid
understating the Hurst parameters shortens it and overstating them lengthens it. At
$\delta=-0.02$ the coverage of the nominal $95\%$ interval is $0.933$ at $N=100$ and
$0.926$ at $N=500$; at $\delta=+0.02$ it is $0.962$ and $0.966$. A practical rule
follows: when $\tH$ is uncertain, erring towards a larger $b_N(\delta)$, which for $h<1$
means rounding the Hurst estimates upwards, yields a conservative interval for $\mu$.
\subsection{Graphical analysis}

\begin{figure}[htbp]
\centering
\includegraphics[width=\textwidth]{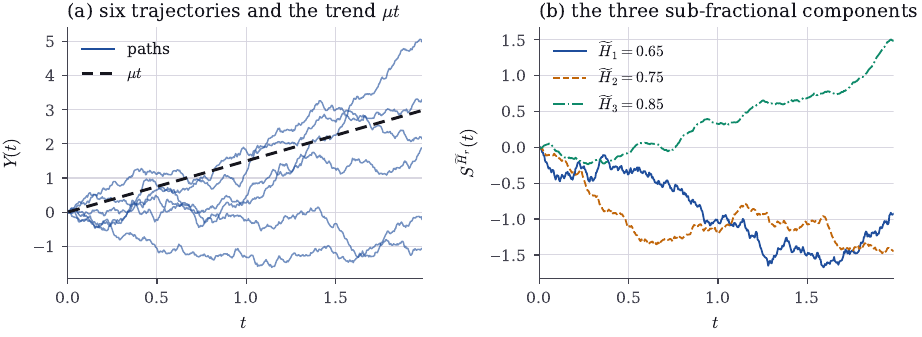}
\caption{The model. (a) Six independent trajectories of \eqref{eq:model} in Scenario 2
with $N=500$, $h=1/252$, and the trend $\mu t$ (dashed). (b) The three sub-fractional
components of one trajectory.}
\label{fi:paths}
\end{figure}

\begin{figure}[htbp]
\centering
\includegraphics[width=\textwidth]{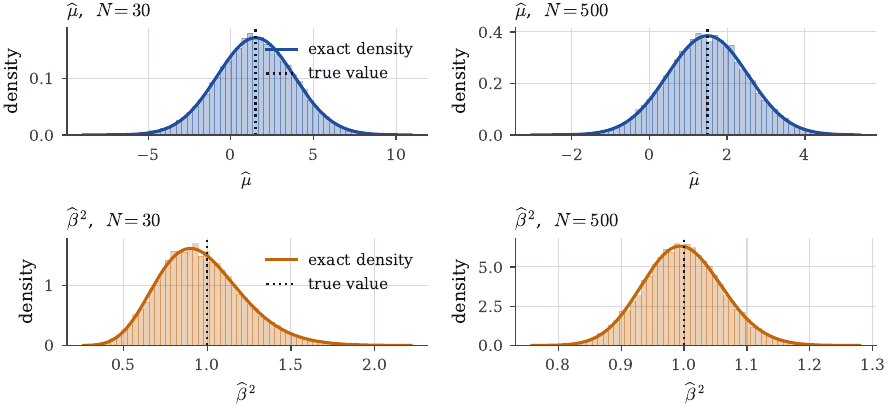}
\caption{Histograms of $\widehat\mu$ (top) and $\widehat\beta^{\,2}$ (bottom) over
$20\,000$ replications of Scenario 2 with $h=1/252$, at $N=30$ (left) and $N=500$
(right), with the exact densities of Theorem \ref{th:exactlaw} superimposed. Dotted
vertical lines mark the true parameter values.}
\label{fi:hist}
\end{figure}

\begin{figure}[htbp]
\centering
\includegraphics[width=\textwidth]{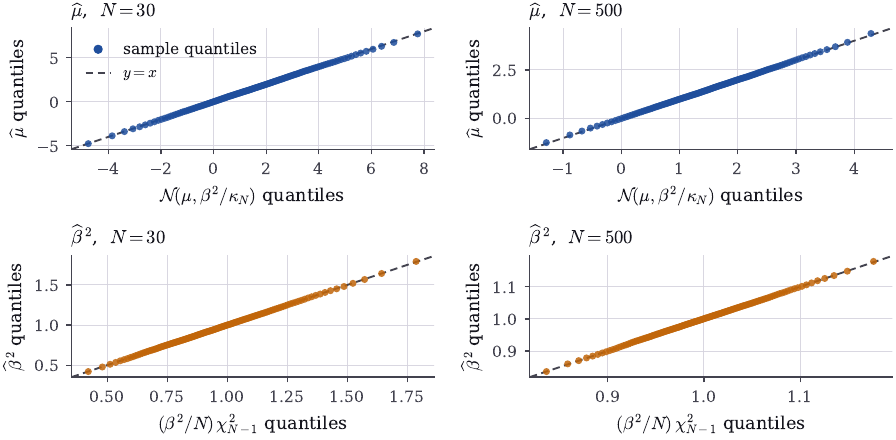}
\caption{Quantile--quantile plots against the exact reference laws of
Theorem \ref{th:exactlaw}, for the same replications as in Figure \ref{fi:hist}.}
\label{fi:qq}
\end{figure}

Figure \ref{fi:paths} illustrates the model. Panel (a) shows six independent trajectories
of \eqref{eq:model} in Scenario 2 with $N=500$ and $h=1/252$, together with the trend
$\mu t$; panel (b) displays the three sub-fractional components of one trajectory and
shows how the roughness decreases with $\tH_r$.

Figures \ref{fi:hist} and \ref{fi:qq} assess Theorem \ref{th:exactlaw} graphically. In
Figure \ref{fi:hist} the histograms of $\widehat\mu$ and $\widehat\beta^{\,2}$ are
superimposed not on a fitted Gaussian density but on the \emph{exact} densities predicted
by the theorem, namely $\Norm(\mu,\beta^{2}/\kappa_N)$ and the density of
$(\beta^{2}/N)\chi^{2}_{N-1}$. The agreement is complete at $N=500$ and equally complete
at $N=30$, where the pronounced right skewness of $\widehat\beta^{\,2}$, which would
appear as a failure of normality, is precisely the skewness of $\chi^{2}_{29}$.
Figure \ref{fi:qq} makes the same point through quantile--quantile plots against the
exact reference laws: the four plots are straight lines at $N=30$ as much as at $N=500$.

\clearpage

\section{Discussion and concluding remarks}\label{se:disc}

For a linear drift observed together with a superposition of independent sub-fractional
Brownian motions sharing a common scale, likelihood inference is exact. The maximum
likelihood estimators of the drift and of the scale are available in closed form; their
joint law is known at every sample size, with $\widehat\mu$ Gaussian,
$N\widehat\beta^{\,2}/\beta^{2}$ chi-square with $N-1$ degrees of freedom and the two
independent; and confidence intervals and tests of exact level follow, together with
complete sufficiency, minimum variance unbiasedness and attainment of the Cram\'er--Rao
bound. The asymptotic theory, mean square and strong consistency and asymptotic
normality, is then a consequence rather than a substitute. This is the practical
difference from the existing treatment of the one-component model \cite{KL15}, where
Berry--Ess\'een bounds are required precisely because the finite-sample law is not
available. The simulations of Section \ref{se:sim} reproduce these statements to Monte
Carlo accuracy, at $N=30$ as much as at $N=500$.

Several features of the analysis deserve emphasis. First, the model extends the
sub-fractional framework to an arbitrary finite number of independent components, which
enlarges the range of accessible covariance structures while preserving Gaussianity and
hence tractability of the likelihood; the process itself is the one introduced in
\cite{SSZ26}, and the single-component model of \cite{KL15} is the case $m=1$.

Second, the entire structure is compressed into the single scalar
$\kappa_N=\btau^{\top}\Sig^{-1}\btau$. Once the noise covariance is known up to a positive factor,
whitening reduces the model to a Gaussian linear model with one regressor
(Remark \ref{re:glm}), and the number of components, the Hurst vector and the sampling
step are all absorbed into $\kappa_N$. A practitioner who can form $\Sig$ and compute
$\kappa_N$ can therefore carry out exact inference with no further analysis of the model.
The same observation delimits the theory: with
component-specific scales the covariance is no longer a known matrix times an unknown
constant, and exactness is lost.

Third, the absence of stationary increments is what separates this model from its
fractional counterpart at the technical level. It removes the Toeplitz reduction of
$\kappa_N$, the spectral density and with it the classical route to a matching lower bound on
$\Var(\widehat\mu)$, and it makes even the nonsingularity of $\Sig$ a question about
local nondeterminism rather than an elementary computation. The representation of
sub-fractional Brownian motion as the even part of a two-sided fractional Brownian
motion, \eqref{eq:evenpart}, is what restores tractability at that point, and it is worth
noting that the same representation is the natural tool for several of the open questions
below.

Fourth, the two parameters behave very differently. Inference about the scale is free of
the model: after whitening, the sampling distribution of $\widehat\beta^{\,2}$ depends on
$N$ alone, so the chi-square interval has exact level whatever the dependence structure.
Inference about the drift is not free; by \eqref{eq:varbound} what can be learned about
$\mu$ is decided by the length of the observation window and by the strongest
long-range dependence present, and refining the mesh on a fixed window does not help.
Section \ref{sse:misspec} shows that the two also inherit the conditioning on $\tH$ very
differently: the Student interval for $\mu$ tolerates a small error in the Hurst
vector and loses level in a direction one can predict, whereas the chi-square interval
for $\beta^{2}$ does not, since $\beta^{2}$ and $\tH$ are nearly confounded.

Several directions remain open. The most pressing is a matching lower bound for
$\Var(\widehat\mu^{(N)})$, which would turn the order in \eqref{eq:varbound} into an
exact rate; Remark \ref{re:nosharp} explains why the stationary-increment argument is
unavailable and why \eqref{eq:evenpart} is the natural substitute. A second is to relax
the common-scale restriction, where Remark \ref{re:glm} indicates what is at stake.
The equidistant sampling assumption, by contrast, enters only through the arithmetic at
the end of Proposition \ref{pr:varbound}; the exact theory of Section \ref{se:mle} holds
for an arbitrary design.

Beyond these, three questions seem to us the most interesting.

\begin{enumerate}[leftmargin=2.0em,itemsep=1.1ex,topsep=1.0ex]

\item \emph{Subordination.} Replace physical time by an independent subordinator
$\{G_t,\,t\ge0\}$, a Gamma or an inverse $\alpha$-stable process, and consider
$\widetilde Y(t)=Y(G_t)$. The correlation structure of such time-changed mixed models is
known \cite{AlM20,Mli23a,Mli23b}; the inferential question is not. The time-changed
process is no longer Gaussian, the likelihood is not available in closed form, and none
of the arguments of Section \ref{se:mle} survives. The Brownian case is encouraging: the
conditional maximum likelihood estimator of the drift is $\widetilde Y(t_N)/G_{t_N}$,
and for a Gamma subordinator with variance rate $\nu$ one obtains the exact law
$\sqrt{t_N}\,(\widehat\mu_N-\mu)\sim t_{2t_N/\nu}$. Does an exact law of this kind
persist for $m$ sub-fractional components?

\item \emph{Higher-order sub-fractional components.} The $n$th-order fractional Brownian
motion admits a Hurst index $H\in(n-1,n)$ and so escapes the ceiling $H<1$; its
statistical theory is developed in Chaouch, El Maroufy and El Omari \cite{CEE23} and El
Omari \cite{ElO23}, and is entirely asymptotic, resting on power variations or on the
ergodic theorem. An $n$th-order sub-fractional analogue does not appear to have been
constructed. If it is built through the even-part representation
\eqref{eq:evenpart}, from a two-sided $n$th-order fractional Brownian motion, two
questions arise: does the nondegeneracy of Lemma \ref{le:nondeg} survive, the strong
local nondeterminism invoked in Remark \ref{re:slnd} being unavailable at that order;
and, if it does,
does the exact theory of Section \ref{se:mle} go through verbatim? The second question
may well have an affirmative answer, since Section \ref{se:mle} uses nothing about
$\Sig$ beyond $\mathrm{Cov}(\bY)=\beta^{2}\Sig$ with $\Sig$ known and positive definite.

\item \emph{An unknown Hurst vector.} The theory above takes $\tH$ as known. For the
mixed fractional models, that is for a Brownian motion superposed with one or more
fractional Brownian motions, \cite{MP23,CDM25,CDL25} estimate the Hurst index at high
frequency and establish the optimal rates; no such theory exists for the
multi-sub-fractional model, whose increments are not stationary. Our aim is to estimate
$\tH$ there, jointly with $\mu$ and $\beta^{2}$ or in a first stage, and to quantify the
resulting loss. Section \ref{sse:misspec} indicates that it will be asymmetric: the
interval for $\mu$ degrades slowly, the one for $\beta^{2}$ does not, $\beta^{2}$ and
$\tH$ being nearly confounded on a fine grid.

\end{enumerate}

Taken together, these observations describe a model in which exactness survives the loss
of stationary increments, at the price of a known Hurst vector and a scale common to the
components. The price is real, and Section \ref{sse:misspec} shows that it is paid almost
entirely by $\beta^{2}$, which is nearly confounded with $\tH$, while the interval for
$\mu$ degrades slowly and in a direction one can predict. Neither restriction, however,
is needed for the whitening argument itself; relaxing them is a question about the
covariance structure rather than about the method of proof, and that is what makes the
directions above worth pursuing.


\begin{thebibliography}{99}

\bibitem{AM21}
S.~Alajmi and E.~Mliki,
\emph{Mixed generalized fractional Brownian motion},
J.\ Stoch.\ Anal.\ \textbf{2} (2021), no.~2, Article~2.

\bibitem{AlM20}
S.~Alajmi and E.~Mliki,
\emph{On the mixed fractional Brownian motion time changed by inverse
$\alpha$-stable subordinator},
Appl.\ Math.\ Sci.\ \textbf{14} (2020), no.~16, 755--763.

\bibitem{BGT04}
T.~Bojdecki, L.~G. Gorostiza, and A.~Talarczyk,
\emph{Sub-fractional Brownian motion and its relation to occupation times},
Statist.\ Probab.\ Lett.\ \textbf{69} (2004), no.~4, 405--419.

\bibitem{CDL25}
C.~H. Chong, T.~Delerue, and G.~Li,
\emph{When frictions are fractional: rough noise in high-frequency data},
J.\ Amer.\ Statist.\ Assoc.\ \textbf{120} (2025), no.~551, 1531--1544.

\bibitem{CDM25}
C.~H. Chong, T.~Delerue, and F.~Mies,
\emph{Rate-optimal estimation of mixed semimartingales},
Ann.\ Statist.\ \textbf{53} (2025), no.~1.

\bibitem{CEE23}
H.~Chaouch, H.~El~Maroufy, and M.~El~Omari,
\emph{Statistical inference for models driven by $n$-th order fractional Brownian
motion},
Theory Probab.\ Math.\ Statist.\ \textbf{108} (2023), 29--43.

\bibitem{Dah89}
R.~Dahlhaus,
\emph{Efficient parameter estimation for self-similar processes},
Ann.\ Statist.\ \textbf{17} (1989), no.~4, 1749--1766;
correction, \emph{ibid.}\ \textbf{34} (2006), no.~2, 1045--1047.

\bibitem{ElO23}
M.~El~Omari,
\emph{Parameter estimation for $n$th-order mixed fractional Brownian motion with
polynomial drift},
J.\ Korean Statist.\ Soc.\ \textbf{52} (2023), 450--461.

\bibitem{HNXZ11}
Y.~Hu, D.~Nualart, W.~Xiao, and W.~Zhang,
\emph{Exact maximum likelihood estimator for drift fractional Brownian motion at discrete
observation},
Acta Math.\ Sci.\ Ser.\ B \textbf{31} (2011), no.~5, 1851--1859.

\bibitem{Jan97}
S.~Janson,
\emph{Gaussian Hilbert Spaces},
Cambridge Tracts in Mathematics, vol.~129, Cambridge University Press, Cambridge, 1997.

\bibitem{KL15}
N.~Kuang and B.~Liu,
\emph{Parameter estimations for the sub-fractional Brownian motion with drift at discrete
observation},
Braz.\ J.\ Probab.\ Stat.\ \textbf{29} (2015), no.~4, 778--789.

\bibitem{LC98}
E.~L. Lehmann and G.~Casella,
\emph{Theory of Point Estimation}, 2nd ed.,
Springer Texts in Statistics, Springer, New York, 1998.

\bibitem{LR05}
E.~L. Lehmann and J.~P. Romano,
\emph{Testing Statistical Hypotheses}, 3rd ed.,
Springer Texts in Statistics, Springer, New York, 2005.

\bibitem{MS23}
H.~Maleki Almani and T.~Sottinen,
\emph{Multi-mixed fractional Brownian motions and Ornstein--Uhlenbeck processes},
Mod.\ Stoch.\ Theory Appl.\ \textbf{10} (2023), no.~4, 343--366.

\bibitem{Men10}
I.~Mendy,
\emph{On the local time of sub-fractional Brownian motion},
Ann.\ Math.\ Blaise Pascal \textbf{17} (2010), no.~2, 357--374.

\bibitem{MRS17}
Y.~Mishura, K.~Ralchenko, and S.~Shklyar,
\emph{Maximum likelihood drift estimation for Gaussian process with stationary
increments},
Austrian J.\ Stat.\ \textbf{46} (2017), no.~3--4, 67--78.

\bibitem{MV15}
Y.~Mishura and G.~Voronov,
\emph{Construction of maximum likelihood estimator in the mixed fractional--fractional
Brownian motion model with double long-range dependence},
Mod.\ Stoch.\ Theory Appl.\ \textbf{2} (2015), no.~2, 147--164.

\bibitem{Mli23a}
E.~Mliki,
\emph{On the fractional mixed fractional Brownian motion time changed by inverse
$\alpha$-stable subordinator},
Glob.\ Stoch.\ Anal.\ \textbf{10} (2023), no.~1, 1--8.

\bibitem{Mli23b}
E.~Mliki,
\emph{Correlation structure of time-changed generalized mixed fractional Brownian
motion},
Fractal Fract.\ \textbf{7} (2023), no.~8, 591.

\bibitem{MP23}
F.~Mies and M.~Podolskij,
\emph{Estimation of mixed fractional stable processes using high-frequency data},
Ann.\ Statist.\ \textbf{51} (2023), no.~5.

\bibitem{NOL08}
D.~Nualart and S.~Ortiz-Latorre,
\emph{Central limit theorems for multiple stochastic integrals and Malliavin calculus},
Stochastic Process.\ Appl.\ \textbf{118} (2008), no.~4, 614--628.

\bibitem{Pit78}
L.~D. Pitt,
\emph{Local times for Gaussian vector fields},
Indiana Univ.\ Math.\ J.\ \textbf{27} (1978), no.~2, 309--330.

\bibitem{RY23}
K.~Ralchenko and M.~Yakovliev,
\emph{Asymptotic normality of parameter estimators for mixed fractional Brownian motion
with trend},
Austrian J.\ Stat.\ \textbf{52} (2023), 127--148.

\bibitem{RY24}
K.~Ralchenko and M.~Yakovliev,
\emph{Parameter estimation for fractional mixed fractional Brownian motion based on
discrete observations},
Mod.\ Stoch.\ Theory Appl.\ \textbf{11} (2024), no.~1, 1--29.

\bibitem{Sgh13}
A.~Sghir,
\emph{The generalized sub-fractional Brownian motion},
Commun.\ Stoch.\ Anal.\ \textbf{7} (2013), no.~3, 373--382.

\bibitem{SSZ26}
F.~Shokrollahi, T.~Sottinen, and M.~Zili,
\emph{Multi-mixed sub-fractional Brownian motion and Ornstein--Uhlenbeck processes},
AIMS Math.\ \textbf{11} (2026), no.~2, 3464--3498.

\bibitem{Tud07}
C.~Tudor,
\emph{Some properties of the sub-fractional Brownian motion},
Stochastics \textbf{79} (2007), no.~5, 431--448.

\bibitem{WCT21}
W.~Wang, G.~Cai, and X.~Tao,
\emph{Pricing geometric Asian power options in the sub-fractional Brownian motion
environment},
Chaos Solitons Fractals \textbf{145} (2021), 110754.

\bibitem{XZZ11}
W.~Xiao, W.~Zhang, and X.~Zhang,
\emph{Maximum-likelihood estimators in the mixed fractional Brownian motion},
Statistics \textbf{45} (2011), no.~1, 73--85.

\bibitem{Zil18}
M.~Zili,
\emph{On the generalized fractional Brownian motion},
Math.\ Models Comput.\ Simul.\ \textbf{10} (2018), no.~6, 759--769.

\end{thebibliography}
\end{document}